\documentclass[12pt]{article}
\usepackage{amssymb,amsfonts,amsmath,amsthm}
\usepackage[normalem]{ulem}
\usepackage{pdfsync}
\usepackage{url}
\usepackage{color}
\usepackage{tikz, sidecap}
\newcommand{\witi}{\widetilde}
\newcommand{\fracd}[2]{\frac {\displaystyle #1}{\displaystyle #2 }}

\newcommand{\nn}{{\mathbb N}}

\newcommand{\rr}{{\mathbb R}}

\newcommand{\zz}{{\mathbb Z}}

\newcommand{\calf}{{\mathcal F}}
\newcommand{\calg}{{\mathcal G}}

\newcommand{\veps}{\varepsilon}

\newcommand{\beq}{\begin{eqnarray*}}
\newcommand{\feq}{\end{eqnarray*}}
\newcommand{\beqn}{\begin{eqnarray}}
\newcommand{\feqn}{\end{eqnarray}}
\newcommand{\as}{\mbox{\rm a.\,s.}}

\newcommand{\one}{\mbox{\large \bf 1}}
\newtheorem{theorem}{Theorem}
\newtheorem{lemma}[theorem]{Lemma}

\newtheorem*{theorem*}{Theorem}
\newtheorem*{corollary*}{Corollary}
\newtheorem{proposition}[theorem]{Proposition}

\newtheorem{remark}[theorem]{Remark}

\newtheorem*{cond-b}{Condition~$\mathbf{B}$}
\newtheorem*{cond-c}{Condition~$\mathbf{C}_\kappa$}
\newtheorem*{cond-cl}{Condition~$\mathbf{C(\kappa,L)}$}

\begin{document}
\title{Relative growth of the partial sums \\ of certain random Fibonacci-like sequences}
\author{Alexander~Roitershtein\thanks{Dept. of Mathematics, Iowa State University, Ames, IA 50011, USA; e-mail: roiterst@iastate.edu}
\and
Zirou Zhou \thanks{Dept. of Mathematics, Iowa State University, Ames, IA 50011, USA; e-mail: zzhou@iastate.edu}
}
\date{January 19, 2017; Revised August 8, 2017}
\maketitle
\begin{abstract}
We consider certain Fibonacci-like sequences $(X_n)_{n\geq 0}$ perturbed with a random noise. Our main result is
that $\frac{1}{X_n}\sum_{k=0}^{n-1}X_k$ converges in distribution, as $n$ goes to infinity, to a random variable $W$ with Pareto-like distribution tails. We show that $s=\lim_{x\to \infty} \frac{-\log P(W>x)}{\log x}$ is a monotonically decreasing characteristic of the input noise, and hence can serve as a measure of its strength in the model. Heuristically,
the heavy-taliped limiting distribution, versus a light-tailed one with $s=+\infty,$ can be interpreted as an evidence supporting the idea that the noise is ``singular" in the sense that it is ``big" even in a ``slightly" perturbed sequence.
\end{abstract}

\section{Introduction and statement of the main result}
\label{intro}
Let $(\eta_n)_{n\geq 0}$ be a sequence of independent Bernoulli random variables with $P(\eta_n=1)=1-\veps,$ $P(\eta_n=0)=\veps$ for some $\veps\in(0,1).$  We consider a sequence $(X_n)_{n\geq 0}$ of real-valued random variables generated by the recursion
\beqn
\label{xdef}
X_{n+1}=aX_n+b\eta_{n-1}X_{n-1},\qquad n\in\nn,
\feqn
with the initial conditions $X_0=1,$ $X_1=a,$ where
\beqn
\label{conda}
a\in(0,1)\qquad \mbox{\rm and} \qquad b>1-a
\feqn
are given deterministic constants.
\par
The above construction is inspired by the models considered in \cite{wdisorder}. The sequence $X_n$ can be thought as a perturbation with noise of its deterministic counterpart, which is defined through the recursion equation
\beq
Z_{n+1}=aZ_n+bZ_{n-1}
\feq
and the initial conditions $Z_0=1,$ $Z_1=a.$ Throughout the paper we are interested in the dependence of model's characteristics on the parameter $\veps$ that varies while the recursion coefficients $a,b$ are maintained fixed.
\par
It is not hard to check that $\lim_{n\to\infty} \frac{1}{Z_n}\sum_{k=0}^{n-1}Z_k=(\lambda_1-1)^{-1},$ where $\lambda_1$ is a constant
defined below in \eqref{lambdas}. In this paper we are primarily concerned with the asymptotic behavior of the following sequence
\beqn
\label{win}
W_n:=\frac{1}{X_n}\sum_{k=0}^{n-1}X_k, \qquad n\in\nn,
\feqn
which describes the rate of growth of the partial sums relatively to the original sequence $X_n.$ Our main result is stated in the following theorem. Intuitively, it can be interpreted as a saying that while adding more noise to the input by increasing the value of $\veps$ yields more noise in the output sequence
$W_n,$ the noise remains large for all, even arbitrarily small, values of the parameter $\veps>0$ in some rigorous sense.
\begin{theorem}
\label{main}
Let $W_n$ be defined in \eqref{win}. Then the following holds true:
\item [(a)] There exists $\veps^*\in (0,1)$ such that
\begin{itemize}
\item [(i)] If $\veps\in (0,\veps^*),$ then $W_n$ converges in distribution, as $n$ goes to infinity, to a non-degenerate random variable $W^{(\veps)}.$
\item [(ii)] If $\veps\in [\veps^*,1),$ then $\lim_{n\to\infty}P(W_n>x)=1$ for any $x>0,$ that is $W^{(\veps)}=+\infty.$
\item [(iii)] For any $\veps \in (0,\veps^*),$ there exist reals $s_\veps\in (0,\infty)$ and $K_\veps\in (0,\infty)$ such that
    \beq
    \lim_{x\to\infty}P(W^{(\veps)}>x)x^{s_\veps}=K_\veps.
    \feq
\end{itemize}
\item [(b)] Furthermore, $s_\veps$ is a continuous strictly decreasing function of $\veps$ on $(0,\veps^*),$ and
\beqn
\label{limits}
\lim_{\veps\downarrow 0}s_\veps=\infty \qquad \mbox{\rm while}\qquad \lim_{\veps\uparrow \veps^*}s_\veps=0.
\feqn
\end{theorem}
$\mbox{}$
\par
The specific choice of the initial values $X_0=1$ and $X_1=a$ is technically convenient, but is not essential. In particular, while asserting it ultimately yields part {\em (a)} of Lemma~\ref{ratios},  changing it wouldn't affect part {\em (b)} of the lemma. Theorem~\ref{main} remains valid for an arbitrary pair $(X_0,X_1)$ of positive numbers. See Remark~\ref{essay} in Section~\ref{proofm} for details. To extend Theorem~\ref{main} to a linear recursion \eqref{xdef} under a more general than \eqref{conda} assumption $a\neq 0,$ $b>0,$ one can consider $\witi X_n=\theta^{-n} X_n$ with an arbitrary $\theta\in\rr$ such that $a\theta>0$ and $2|a|<2|\theta|<|a|+\sqrt{a^2+4b}.$ The new sequence $\witi X_n$ satisfies the recursion $\witi X_{n+1}=\witi a\witi X_n+\witi b\eta_{n-1}\witi X_{n-1}$ with $\witi a=a/\theta<1$ and $\witi b=b/\theta^2>1-\witi a.$ Some other readily available extensions of Theorem~\ref{main} are discussed in Section~\ref{remarks} below.
\par
The proof of Theorem~\ref{main} is given in Section~\ref{proofm} below. Note that the theorem implies that the limiting distribution $W^{(\veps)}$ has power
tails as long as it is finite and non-degenerate. We remark that additional properties of the constants $\veps^*$ and $s_\veps$ can be inferred from the auxiliary results discusses in Section~\ref{proofm} below. In particular, see Proposition~\ref{prop} which provides some information on the relation of $s_\veps$ to the Lyapunov exponent and the moments of the reciprocal sequence $X_n^{-1}.$
\par
For an integer $n\geq 0,$ let
\beqn
\label{arn}
R_n=\frac{X_n}{X_{n+1}}.
\feqn
The sequence $R_n$ forms a Markov chain since \eqref{xdef} is equivalent to $R_n=(a+b\eta_{n-1}R_{n-1})^{-1}.$ Notice that, since $X_0=1,$ for $n\in\nn$ we have $X_n^{-1}=\prod_{k=0}^{n-1} R_n$ and
\beqn
\label{recur}
W_{n+1}=R_nW_n+R_n \quad \mbox{\rm or, equivalently,}\quad (W_{n+1}+1)=R_n(W_n+1)+1.
\feqn
The proof of the assertion {\em (a)-(iii)} of Theorem~\ref{main} is carried out by an adaption of the technique used in \cite{stable} to obtain an
extension of Kesten's theorem \cite{goldie,kesten} for linear recursions
with i.\,i.\,d. coefficients to a setup with Markov-dependent coefficients. More specifically, to prove that the distribution of $W^{(\veps)}$ is asymptotically power-tailed, we verify in Section~\ref{proofm} that Markov chain $R_n$ satisfies Assumption~1.5 in \cite{stable}. This allows us to borrow key auxiliary results from \cite{stable,mrec} and also use a variation of the underlying regeneration structure argument in \cite{stable}. See Lemma~\ref{stablef} in Section~\ref{proofm} below for details.
\par
The proof of Theorem~\ref{main} relies in particular on the asymptotic analysis of the negative moments of $X_n$ (more specifically, the function $\Lambda_\veps(t)$ defined below in \eqref{istable}). First positive integer moments of $X_n$ can be in principle computed explicitly. We conclude this introduction with the statement of a result which is not directly connected to Theorem~\ref{main}, but might be useful, for instance,
for the statistical analysis of the sequence $X_n.$ Here and throughout this paper we use the notation $E_P$ to denote the expectation operator under the probability law $P$ (in order to distinguish it from the expectation $E_Q,$
where $Q$ is introduced in Section~\ref{prelim} below). For $\veps\in[0,1],$ let
\beqn
\label{lambdase}
\lambda_{\veps,1}=\frac{a+\sqrt{a^2+4b(1-\veps)}}{2}>0\quad\mbox{\rm and} \quad \lambda_{\veps,2}=\frac{a-\sqrt{a^2+4b(1-\veps)}}{2}<0
\feqn
denote the roots of the characteristic equation $\lambda^2=a\lambda+b(1-\veps).$ We have:
\begin{proposition}
\label{stat}
For any integer $n\geq 0,$
\item [(a)]  We have $E_P(X_n)=\fracd{\lambda_{\veps,1}^{n+1}-\lambda_{\veps,2}^{n+1}}{\lambda_{\veps,1}-\lambda_{\veps,2}}.$
In particular, $\lim_{n\to\infty}\frac{1}{n}\log E_P(X_n)=\lambda_{\veps,1}.$
\item [(b)]  We have $\lim_{n\to\infty}\frac{1}{n}\log E_P(X_n)=a\lambda_{\veps,1}+b(1-\veps).$ More precisely,
\beq
E_P(X_n^2)=c_1[a\lambda_{\veps,1}+b(1-\veps)]^n+c_2[a\lambda_{\veps,2}+b(1-\veps)]^n-\frac{2(-b)^{n+1}(1-\veps)^n}{4b(1-\veps)+a^2},
    \feq
    where are the constants $c_1$ and $c_2$ are chosen in a manner consistent with the initial conditions $X_0=1$ and $X_1=a.$
\item [(c)]  Letting $U_{n,k}:=E_P(X_nX_{n+k})$ for an integer $k\geq 0,$
\beq
U_{n,k}=d_{n,1}\lambda_{\veps,1}^k+d_{n,2}\lambda_{\veps,2}^k,
\feq
where are the constants $d_{n,1}$ and $d_{n,2}$ are chosen in a manner consistent with the initial conditions $U_{n,0}=E_P(X_n^2)$ and
\beq
U_{n,1}=E_P(X_nX_{n+1})=\frac{1}{a}\bigl[E_P(X_{n+1}^2)-b(1-\veps)E_P(X_n^2)-(-b)^{n+1}(1-\veps)^n\bigr].
\feq
In particular, for any $k\in\nn$ we have $\lim_{n\to\infty}\frac{1}{n}\log E_P(X_nX_{n+k})=a\lambda_{\veps,1}+b(1-\veps).$
\end{proposition}
$\mbox{}$
\par
The rest of the paper is organized as follows. Section~\ref{prelim} contains a preliminary discussion and an auxiliary monotonicity result (with respect to the parameter $\veps$) about the Lyapunov constant of the sequence $X_n.$ The proof of Theorem~\ref{main} is included in Section~\ref{proofm}. The proof of Proposition~\ref{stat} is deferred to Section~\ref{stats}. Finally, section~\ref{remarks} contains some concluding remarks regarding possible extensions of the results in Theorem~\ref{main}.
\section{Preliminaries: Lyapunov constant of $X_n$}
\label{prelim}
This section includes a preliminary discussion which is focused on the random variable $R_n$ defined in \eqref{arn} and the Lyapunov constant
$\gamma(\veps)$ introduced below. The main purpose here is to obtain a monotonicity result in Proposition~\ref{renewal}. The coupling construction employed
to prove Proposition~\ref{renewal} is also used in Section~\ref{proofm}, to carry out the proof of Propositions~\ref{prop} and~\ref{prop3}.
\par
Recall $\lambda_{\veps,1}$ and $\lambda_{\veps,2}$ from \eqref{lambdase}. In order to simplify the notation, denote
\beqn
\label{lambdas}
\lambda_1:=\lambda_{0,1}=\frac{a+\sqrt{a^2+4b}}{2}\qquad\mbox{\rm and} \qquad \lambda_2:=\lambda_{2,0}=\frac{a-\sqrt{a^2+4b}}{2}.
\feqn
Notice that the condition $a+b>1$ ensures $\lambda_1>1.$ Using the initial conditions $Z_0=1$ and $Z_1=a,$ one can verify that
\beqn
\label{ze}
Z_n=\frac{\lambda_1^{n+1}-\lambda_2^{n+1}}{\lambda_1-\lambda_2},\qquad n\geq 0.
\feqn
Using \eqref{ze} one can obtain a Cassini-type identity $Z_{n-1}Z_{n+1}-Z_n^2=b^n(-1)^{n+1}$ (see, for instance, Theorem~5.3 in \cite{fbook} for the original Fibonacci sequence result) and the identity $Z_{n+1}-\lambda_1Z_n=\lambda_2^{n+1}.$ The alternating sign of the right-hand side in these two identities yields
for $k\in\nn,$
\beqn
\label{ordering}
a\leq \frac{Z_{2k-1}}{Z_{2k-2}}<\frac{Z_{2k+1}}{Z_{2k}}<\lambda_1<\frac{Z_{2k+2}}{Z_{2k+1}}<\frac{Z_{2k}}{Z_{2k-1}}\leq \frac{a^2+b}{a}.
\feqn
Recall $R_n$ from \eqref{arn}. The (unique) stationary distribution for the countable, irreducible and aperiodic Markov chain $R_n$ can be obtained as follows. For $n\in\nn,$ let
\beqn
\label{tin}
T_n=\sup\{i\leq n: \eta_i=0\}.
\feqn
Fix any $k\in\nn.$  Then for a positive integer $n>k$ we have
\beq
&&
P\bigl(R_n=Z_{k-1}/Z_k\bigr)=P(T_n=n-k-1)
\\
&&
\qquad
\qquad
=P\bigl(\eta_{n-k-1}=0,\eta_{n-k}=\eta_{n-k+1}=\ldots=\eta_{n-3}=\eta_{n-2}=1\bigr)=\veps(1-\veps)^{k-1}.
\feq
Thus, $\lim_{n\to \infty} P\bigl(R_n=Z_{k-1}/Z_k\bigr)= \veps(1-\veps)^{k-1}.$
The stationary sequence $(R_n)_{n\in\nn}$ can be extended into a double-infinite stationary sequence $(R_n)_{n\in\zz}$ \cite{durrett}.
Let $Q$ denote the law of the time-reversed stationary Markov chain $(R_{-n})_{n\in \zz}.$
We have established the following result:
\begin{lemma}
\label{ratios}
Let $S_k=\frac{Z_{k-1}}{Z_k},$ $k\geq 1.$ Then:
\item [(a)] For all $n\in\zz,$ we have $P\bigl(R_n\in \{S_k:k\in\nn\}\bigr)=1.$
\item [(b)] Furthermore, $Q(R_n=S_k)=\veps(1-\veps)^{k-1}$ for any $n\in\zz,k\in\nn.$
\end{lemma}
$\mbox{}$
\par
Let $\gamma=\gamma(\veps)$ denote the Lyapunov exponent of the sequence $X_n,$ that is
\beqn
\label{gammaf}
\gamma:=\lim_{n\to\infty} \frac{1}{n}\log X_n=\lim_{n\to\infty} \frac{1}{n}E_P(\log X_n)=E_Q\Bigl(\log \frac{1}{R_1}\Bigr),
\quad P-\as~\mbox{\rm and}~Q-\as
\feqn
The existence of the limit along with the identities follow from results in \cite{rmp}.
Taking in account that $Z_0=1,$ \eqref{ordering} implies that
\beqn
\label{gamma}
\gamma&=&-E_Q(\log R_1)=-\sum_{n=1}^\infty \veps(1-\veps)^{n-1}\log S_n=\sum_{n=1}^\infty \veps^2(1-\veps)^{n-1}\log Z_n.
\feqn
The last formula can be compactly written as $\gamma=\veps\cdot E_P(\log Z_T),$ where
\beqn
\label{timet}
T=1+\inf\{k\geq 0: \eta_k=0\}=\inf\{j\geq 1:R_j=1/a\}.
\feqn
It follows from \eqref{gamma} and the fact that $|\log S_n|$ is a bounded sequence,
that $\gamma(\veps)$ is an analytic function of $\veps$ on  $[0,1].$ In particular,
\beqn
\label{bounds}
\lim_{\veps\downarrow 0} \gamma(\veps)=\lim_{n\to\infty} \log S_n=\log \lambda_1
\qquad \mbox{\rm and} \qquad \lim\limits_{\veps\uparrow 1} \gamma(\veps)=a.
\feqn
We remark that the analyticity of $\gamma(\veps)$ on $[0,1)$ follows directly from a general result in \cite{peres}. For recent advances in numerical study of 
the Lyapunov exponent for random Fibonacci sequences see \cite{refer,refer1} and references therein. 
\par
We next prove formally the following intuitive result. Together with \eqref{bounds} it implies the existence of $\veps^*\in (0,1)$
such that $\gamma(\veps)>0$ if and only if $\veps<\veps^*.$ Our interest to this phase transition steams from the result in Lemma~\ref{gsums}
stated below in Section~\ref{proofm}.
\begin{proposition}
\label{renewal}
The function $\gamma(\veps):[0,1]\to\rr$ is strictly decreasing.
\end{proposition}
\begin{proof}
The proof is by a coupling argument. Fix any $\veps\in [0,1)$ and $\veps_1\in (\veps,1].$ Let $(X_n)_{n\geq 0}$ be the sequence
introduced in \eqref{xdef}, and define $(X^{(1)}_n)_{n\geq 0}$ as follows: $X^{(1)}_0=1,$ $X^{(1)}_1=a,$ and
\beq
X^{(1)}_{n+1}=aX^{(1)}_n+b\eta^{(1)}_{n-1}X^{(1)}_{n-1},\qquad n\in\nn,
\feq
where $\eta^{(1)}_n=\min\{\eta_n,\xi_n\}$ and $\xi_n$ are i.\,i.\,d. random variables with the distribution
\beq
\xi_n=\left\{
\begin{array}{lll}
0&\mbox{with probability}&\frac{\veps_1-\veps}{1-\veps}\\
1&\mbox{with probability}&\frac{1-\veps_1}{1-\veps},
\end{array}
\right.
\feq
such that $\xi_n$ is independent of the $\sigma$-algebra $\sigma(X_0,\eta_0,X_1,\eta_1,\ldots,X_{n-1},\eta_{n-1},X_n,X_{n+1})$ for all $n\geq 0.$
Then $P\bigl(\eta^{(1)}_n=0\bigr)=\veps_1,$ $P\bigl(\eta^{(1)}_n=1\bigr)=1-\veps_1,$ and hence the sequence $(X^{(1)}_n)_{n\geq 0}$ is distributed according to the same
law as $(X_n)_{n\geq 0}$ with $\veps_1$ replacing $\veps$ in the definition of $\eta_n.$ To deduce that $\gamma(\veps)$ is a non-increasing function of $\veps,$ observe that by the coupling construction, $\eta_n\geq \eta_n^{(1)}$ for all $n\geq 0,$ and hence, by induction, $X_n\geq X^{(1)}_n$ for all $n\geq 0.$
\par
To conclude the proof of the proposition it remains to show that $\gamma(\veps)$ is strictly decreasing. Toward this end, first observe that for any integer $n\geq 2$ we have
\beq
\zeta_n:=\one\bigl(\eta^{(1)}_n\neq \eta_n\bigr)=\one\bigl(\eta^{(1)}_n=0, \eta_n=1\bigr)=\one\bigl(\eta_n=1,\xi_n=0\bigr),
\feq
where $\one(A)$ denotes the indicator function of the event $A$ and the first equality serves as a definition of $\zeta_n.$ Then, the following is an implication of Lemma~\ref{ratios} and \eqref{ordering} along with the fact (which we have established) that $X_n\geq X^{(1)}_n$ for all $n\geq 0:$
\beqn
\nonumber
\log X_n-\log X^{(1)}_n&=&\log \frac{X_n}{X^{(1)}_n}\geq \log\left[\prod_{k=0}^{n-2} \Bigl(\frac{aX_{n-1}+bX_{n-2}}{aX^{(1)}_{n-1}}\Bigr)^{\zeta_k}\right]
\\
\nonumber
&\geq&
\log \left[\prod_{k=0}^{n-2} \Bigl(\frac{aX_{n-1}+bX_{n-2}}{aX_{n-1}}\Bigr)^{\zeta_k}\right]\geq \log \left[\prod_{k=0}^{n-2} \Bigl(1+\frac{b}{a^2+b}\Bigr)^{\zeta_k}\right]
\\
\label{couplee}
&\geq&
\sum_{k=0}^{n-2} \zeta_k \cdot \log \Bigl(1+\frac{b}{a^2+b}\Bigr).
\feqn
It follows then from \eqref{gammaf} and the law of large numbers that with probability one,
\beq
\gamma(\veps)-\gamma(\veps_1)&=&\lim_{n\to\infty} \frac{1}{n} \bigl(\log X_n-\log X^{(1)}_n\bigr)
\geq\lim_{n\to\infty} \frac{1}{n}\sum_{k=0}^{n-2} \zeta_n \cdot \log \Bigl(1+\frac{b}{a^2+b}\Bigr)
\\
&=& E_P(\zeta_n)\cdot \log \Bigl(1+\frac{b}{a^2+b}\Bigr)
=
(\veps_1-\veps) \cdot \log \Bigl(1+\frac{b}{a^2+b}\Bigr)>0.
\feq
The proof of the proposition is complete.
\end{proof}
\section{Proof of the main result}
\label{proofm}
The purpose of this section is to prove Theorem~\ref{main}. The proof is divided into a sequence of lemmas.
The critical exponent $\veps^*$ is identified and part {\em (a)-(i)} and {\em (a)-(ii)} of the theorem are proved in Proposition~\ref{prop}.
The assertion in part {\em (a)-(iii)} of the theorem is verified in Lemma~\ref{stablef}. Finally, the claim in part {\em (b)} is established in Proposition~\ref{prop3}.
\par
Observe that under the stationary law $Q$ the random variable $W_n=\sum_{k=0}^{n-1} \prod_{j=k}^{n-1} R_j$ has the same distribution
as $\sum_{k=0}^{n-1} \prod_{j=0}^k R_{-j}.$ Therefore one can write $W^{(\veps)}=\sum_{k=0}^\infty \prod_{j=0}^k R_{-j}.$ The following lemma is well-known, see for instance Theorem~2.1.2 (especially display (2.1.6)) and the subsequent Remark in \cite{notes}.
\begin{lemma}
\label{gsums}
For any $\veps\in(0,1)$ we have $P(W^{(\veps)}<\infty)=Q(W^{(\veps)}<\infty)\in\{0,1\}.$ Moreover, $P(W^{(\veps)}<\infty)=1$ if and only if $\gamma(\veps)=-E_Q(\log R_1)>0.$
\end{lemma}
Let $s^*=\sup\{\veps>0: \gamma(\veps)>0\}.$ It follows from Proposition~\ref{renewal}, the limits in \eqref{bounds}, and the continuity of $\gamma(\veps)$ that
\beqn
\label{star}
\veps^*\in(0,1)\qquad \mbox{\rm and} \qquad \gamma(\veps^*)=0.
\feqn
By virtue of Lemma~\ref{gsums}, $\veps^*$ satisfies {\em (a)-(i)} and {\em (a)-(ii)} in the statement of Theorem~\ref{main}. We proceed
with the proof that {\em (a)-(iii)} of the theorem also holds true.
\par
Following \cite{stable, mrec}, we are going to identify the critical exponent $s_\veps$ in the statement of Theorem~\ref{main} as the unique solution to the equation $\Lambda_\veps(s_\veps)=0,$ where for $t\geq 0$ we define
\beq
\Lambda_\veps(t):=\lim_{n\to\infty} \frac{1}{n}\log E_Q\bigl(R_1^t\ldots R_n^t).
\feq
It follows from Lemmas~2.6 and~2.8(a) in \cite{stable} (see especially display (2.11) in \cite{stable}) applied to the forward Markov chain $(R_n)_{n\geq 0}$ that the above limit exists and in fact is not affected by the initial distribution of the Markov chain. In particular we have:
\beqn
\label{istable}
\Lambda_\veps(t)=\lim_{n\to\infty} \frac{1}{n}\log E_P\Bigl(\frac{X_0^t}{X_{n-1}^t}\Bigr)=\lim_{n\to\infty} \frac{1}{n}\log E_P\Bigl(\frac{1}{X_n^t}\Bigr)
\feqn
The following proposition is a key ingredient in the proof of Theorem~\ref{main}. Note that for any $\veps\in[0,1],$ $\Lambda_\veps(0)=0$
and, by virtue of the Cauchy-Schwarz inequality, $\Lambda_\veps(t)$ is convex on $[0,\infty).$ In particular, the one-sided derivative $\Lambda'_\veps(0):=\lim_{t\downarrow 0} \Lambda_\veps(t)/t$ is well-defined.
\begin{proposition}
\label{prop}
Let $\veps^*\in(0,1)$ be defined in \eqref{star}. Then the following three statements are equivalent for $\veps\in (0,1):$
\begin{itemize}
\item [(i)] $\gamma(\veps)>0,$ that is $\veps\in(0,\veps^*).$
\item [(ii)] $\Lambda'_\veps(0)<0.$
\item [(iii)] There exists a unique $s_\veps>0$ such that $\Lambda_\veps(s_\veps)=0.$
\item [(iv)] $W^{(\veps)}$ is a $P$-a.\,s. finite and non-degenerate random variable.
\end{itemize}
\end{proposition}
\begin{proof}
$\mbox{}$
\\
$(i)\Rightarrow (ii)$ If $\gamma(\veps)> 0,$ the ergodic theorem implies that for $T$ defined in \eqref{timet}, $Q$-a.\,s.,
\beq
E_Q\Bigl(\sum_{k=0}^{T-1} \log R_k\Bigr)&=& E_Q(T)\cdot \lim_{n\to\infty} \frac{1}{n}\sum_{k=1}^n \log R_k=E_Q(T)\cdot E_Q(\log R_1)
=
-\veps^{-1}\cdot \gamma(\veps)<0.
\feq
Hence \cite{goldie, kesten} there exists a unique $s>0$ such that $E_P\bigl(X_T^{-s}\bigr)=1.$ It can be shown (see, for instance, display (2.43) in \cite{stable}) that this implies $\Lambda_\veps(s)=0$ and hence $\Lambda'_\veps(0)<0.$
\\
$(ii)\Rightarrow (i)$ Jensen's inequality implies that $\gamma_\veps \geq -t \Lambda_\veps(t),$ and hence $\gamma(\veps)> 0$ if $\Lambda'(0)<0.$
\\
$(ii)\Leftrightarrow (iii)$ For any $t>0,$ we have
\beq
E_P(X_n^{-t})\geq E_P\Bigl(X_n^{-t}\prod_{k=2}^n (1-\eta_k)\Bigr)\geq \veps^{n-1}a^{-t(n-1)},
\feq
and hence $\lim_{t\to\infty} \Lambda_\veps(t)=+\infty.$ Since $\Lambda_\veps(t)$ is a convex function with $\Lambda_\veps(0)=0,$ this proves the implication $(ii)\Leftrightarrow (iii)$ (for an illustration, see Fig.~\ref{figure} below).
\\
$(i)\Leftrightarrow (iv)$ This is the content of Lemma~\ref{gsums}.
\\
The proof of the proposition is complete.
\end{proof}
$\mbox{}$
\\
\begin{SCfigure}[][!ht]
\label{figure}
\begin{tikzpicture}[scale=0.7]
\draw [->] (1,9) -- (15,9)node[below]{$t$};
\draw [->] (3,6) -- (3,16)node[left]{$\Lambda_\veps(t)$};
\draw[domain=3:12] plot (\x,{0.1*(\x-6)^2+8.1});
\draw[domain=3:12] plot (\x,{0.1*(\x-7)^2+7.4});
\draw[domain=3:12] plot (\x,{0.1*(\x-5)^2+8.6});
\draw [dashed] [domain=3:12] plot (\x,{0.08*(\x-3)^2+9.0});
\draw (12,9.95)node[right]{$\veps_1$};
\draw (12,11.8)node[right]{$\veps_2$};
\draw (12,13.55)node[right]{$\veps_3$};
\draw (12,15.6)node[right]{$\veps^*$};
\draw (7.2,9)node[below]{$s_{\veps_3}$};
\draw (9.2,9)node[below]{$s_{\veps_2}$};
\draw (11.2,9)node[below]{$s_{\veps_1}$};
\draw [fill] (7,9) circle [radius=.035];
\draw [fill] (9,9) circle [radius=.035];
\draw [fill] (11,9) circle [radius=.035];
\draw [dashed] (3, 9) -- (4.5,6);
\draw (4.6,6.1)node[right]{\footnotesize $\Lambda_0(t)=-t\log \lambda_1$};
\draw (3,9.9)node[right]{\footnotesize $\Lambda_{\veps^*}'(0)=0$};
\end{tikzpicture}
\caption{\small Sketch of the graph of the convex function $\Lambda_\veps(t)$ for 3 increasing parameter values $\veps_1<\veps_2<\veps_3$ within the range $(0,\veps^*)$ and the extremal parameter values $\veps=0,$ $\veps=\veps^*.$ }
\end{SCfigure}
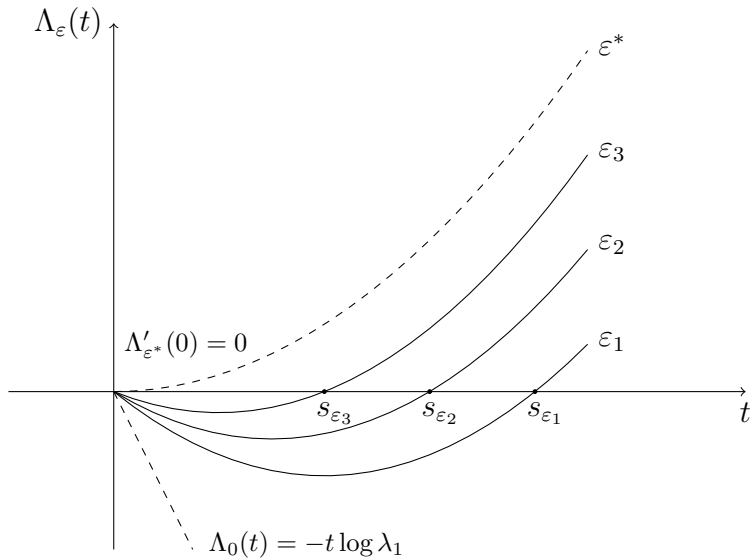
$\mbox{}$
\\
Using the notation introduced in the statement of Lemma~\ref{ratios}, transition kernel of the time-reversed Markov chain $R_{-n}$ on the state space $\{S_i:i\in\nn\}$
can be written as follows:
\beq
H(i,j)&:=&Q(R_n=S_j|R_{n+1}=S_i)=\frac{Q(R_{n+1}=S_i|R_n=S_j)Q(R_n=S_j)}{Q(R_n=S_i)}
\\
&=&
\left\{
\begin{array}{ll}
\veps(1-\veps)^{j-1}&\mbox{\rm if}~i=1\\
1&\mbox{\rm if}~i=j+1\\
0&\mbox{\rm otherwise.}
\end{array}
\right.
\feq
Unfortunately, the infinite matrix $H(i,j)$ doesn't satisfy the conditions imposed in \cite{stable,mrec} or \cite{msee}.
More precisely, the kernel $H$ doesn't satisfy the following strong Doeblin condition: $H^m(i,j) \geq c\mu(j)$ for some $m\in\nn,$ $c>0,$
a probability measure $\mu$ on $\nn,$ and all $i,j\in \nn.$ However, one can exploit the fact that transition kernel of the Markov chain $R_n$ does satisfy Doeblin's condition with $m=1,$ $c=\veps,$ and $\mu=\delta_1,$ the degenerate distribution concentrated on $j=1.$
The proof of the following lemma is a mixture of arguments borrowed from \cite{stable} and \cite{mrec}. The key technical ingredient of the proof
is the observation that transition kernel of the forward Markov chain $R_n$ satisfies Assumption~1.2 in \cite{mrec}.
\begin{lemma}
\label{stablef}
The claim in part (a)-(iii) of Theorem~\ref{main} holds with $\veps^*$ introduced in \eqref{star}.
\end{lemma}
\begin{proof}
Let
$N_0=0$ and then for $i\in\nn,$
\beq
N_i=\sup\{k<N_{i-1}: R_{-k}=1/a\}.
\feq
Note that the blocks $(R_{_{N_{i+1}+1}},\ldots,R_{_{N_i}})$ are independent and identically distributed for $i\geq 0.$ For $i \geq 0,$ let
\beq
A_i&=&R_{_{N_i}}+R_{_{N_i}}R_{_{N_i-1}}+...+
 R_{_{N_i}}R_{_{N_i-1}} \ldots R_{_{N_{i+1}+2}}R_{_{N_{i+1}+1}}
\\
B_i&=& R_{_{N_i}}R_{_{N_i-1}} \ldots R_{_{N_{i+1}+1}}.
\feq
The pairs $(A_i,B_i),$ $i\geq 0,$ are independent and identically distributed under the law $P.$ Moreover, it follows from \eqref{recur} that
\beq
W^{(\veps)}=A_0+\sum_{n=1}^\infty A_n \prod_{i=0}^{n-1} B_i.
\feq
To prove Lemma~\ref{stablef} we will verify the conditions of the following Kesten's theorem for $(A_i,B_i)_{i\geq 0}$ under the law $P.$
To enable a further reference (see Section~\ref{remarks} below) we quote this theorem in a more general setting (with not necessarily strictly positive coefficients $A_n,$ $B_n$)
than we actually need for the purpose of proving Lemma~\ref{stablef}.
\begin{theorem} \cite{goldie, kesten}
\label{wase}
Let $(A_i,B_i)_{i \geq 0}$ be i.i.d. pairs of real-valued random variables such that
\item[(i)] For some $s>0,$ $E\bigl(|A_0|^s\bigr)=1$ and $E\bigl(|B_0|^s \log^+
|B_0|\bigr)<\infty,$ where $\log^+ x:=\max\{\log x,0\}.$
\item[(ii)] $P(\log |B_0|=\delta \cdot k~\mbox{\rm for some}~k\in\zz|B_0 \neq 0)<1$ for all $\delta>0.$
\par
Let $W=A_0+\sum_{n=1}^\infty A_n \prod_{i=1}^{n-1} B_i.$ Then
\begin{itemize}
\item [(a)]
$\lim\limits_{t \to \infty} t^s P(W>t\bigr)=K_+,$
$\lim\limits_{t \to \infty} t^s P(W < -t\bigr)=K_-$ for some $K_+, K_-\geq 0.$
\item [(b)]
If $P(B_1<0)>0,$ then $K_+=K_-.$
\item [(c)] $K_++K_->0$ if and only if $P\bigl(A_0=(1-B_0)c\bigr)<1$ for all $c \in \rr.$
\end{itemize}
\end{theorem}
Recall $T$ from \eqref{timet}. Observe that $\log B_0=\sum_{k=N_1+1}^0\log R_k$ is
distributed the same as $\sum_{k=1}^T \log R_k=\log Z_1-\log Z_{T+1}.$ Therefore the non-lattice condition {\em (ii)} of the above theorem holds in virtue of \eqref{ze}.
Furthermore, since clearly $P(B_0>1)>0$ and $P(B_0<1)>0,$ we have $P\bigl(A_0=(1-B_0)c\bigr)<1$ for all $c \in \rr.$
\par
It remains to verify condition {\em (i)} of the theorem. Recall $S_k$ introduced in the statement of Lemma~\ref{ratios}. Let
\beq
\witi H(i,j)&:=&Q(R_{n+1}=S_j|R_n=S_i)=
\left\{
\begin{array}{ll}
\veps&\mbox{\rm if}~j=1\\
1-\veps&\mbox{\rm if}~j=i+1\\
0&\mbox{\rm otherwise}
\end{array}
\right.
\feq
be transition kernel of the stationary Markov chain $R_n.$ Between two successive regeneration times $N_i$ the forward chain $R_n$ evolves according to a sub-Markov kernel $\Theta$ given by the equation
\beqn
\label{theta}
\witi H(i,j)=\Theta(i,j)+\veps \one(j=1).
\feqn
That is, for $i,j\in\nn,$
\beq
\Theta(i,j)=Q(R_1=j,N_1>1|R_0=i)=
\left\{
\begin{array}{ll}
1-\veps&\mbox{\rm if}~j=i+1\\
0&\mbox{\rm otherwise.}
\end{array}
\right.
\feq
Further, for any real $s\geq 0$ define the kernels (countable matrices) $\witi H_s(i,j)$ and $\Theta_s(i,j),$ $i,j\in\nn,$ by setting
$\witi H_s(i,j)=\witi H(i,j)R_j^t$ and $\Theta_s(i,j)=\Theta(i,j)R_j^t.$ For an infinite matrix $A$ on $\nn$ and a function $f:\nn\to \rr$
let $Af$ denote the real-valued function on $\nn$ with $(Af)(i):=\sum_{j\in\nn}A(i,j)f(j).$ Since the forward transition kernel $\witi H$ satisfies Assumption~1.2 in \cite{mrec}, it follows that from Proposition~2.4 in \cite{mrec} that for all $s\geq 0:$
\begin{enumerate}
\item There exist a real number $\alpha_s>0$ and a bounded function $f_s:\nn\to\rr$ such that $\inf_{i\in\nn} f_s(i)>0$ and $\witi H_sf=\alpha_sf_s.$
\item There exist a real number $\beta_s>0$ and a bounded function $g_s:\nn\to\rr$ such that $\inf_{i\in\nn} g_s(i)>0$ and $\Theta_sf=\beta_sf_s.$
\item $\beta_s\in(0,\alpha_s).$
\end{enumerate}
Without loss of generality we can use the following normalization for the eigenfunctions:
\beqn
\label{normal}
f_s(1)=1.
\feqn
Furthermore, it follows from Lemma~2.3 in \cite{mrec} that $\alpha_s$ and $\beta_s$ are spectral norms of infinite matrices $\witi H_s$ and $\Theta_s,$ respectively, and hence are uniquely defined. It follows from Proposition~\ref{prop} (see Lemma~2.3 in \cite{mrec}) that $\alpha_{s_\veps}=1.$ In particular, since $\Lambda_\veps$ is a continuous function of $s,$ the spectral radius of $\Theta_s$ (regarded as an operator acting on the space of bounded function on $\nn$ equipped with the sup-norm) is strictly less than one on an interval
$\bigl(0,\witi s_\veps\bigr)$ for some $\witi s_\veps>s_\veps.$ Let $I$ be the infinite unit matrix in $\nn$ and $h:\nn\to\rr$ be a function defined by $h(i)=1$ for all $i\in\nn.$ For any $s\in \bigl(0,\witi s_\veps\bigr),$ and in particular for $s=s_\veps,$ we have:
\beq
E_P(B_0^s)&=&E_P\left( \prod_{k=1}^T R_k\right)=E_P\left( \prod_{k=0}^{T-1} R_k\right)=E_P\bigl[\witi H_s^T(1,1)\bigr]
\\
&=&
\sum_{n=1}^\infty a^{-s}\veps r\Theta_s^{n-1}h(1)=a^{-s}\veps (I-\Theta_s)^{-1}h(1).
\feq
On the other hand, it follows from \eqref{theta} that for any $i\in\nn$ we have
\beq
f_{s_\veps}(i)=\witi H_{s_\veps}f_{s_\veps}(i) = \Theta_{s_\veps} f_{s_\veps}(i)+\veps a^{-{s_\veps}}f_{s_\veps}(1),
\feq
and hence $E_P(B_0^{s_\veps})=f_{s_\veps}(1)=1,$ where for the second identity we used \eqref{normal}.
\par
Finally, adapting (2.45) in \cite{stable} to our framework we obtain for any $s\in \bigl(s_\veps,\witi s_\veps\bigr),$
\beq
E_P(A_0^s)&=& E_P\left[ \left(\sum_{n=1}^\infty \sum_{i=1}^n\prod_{j=0}^{i-1} R_{-j}\cdot \one(N_1=-n) \right)^s \right]
\\
&=&\sum_{n=1}^\infty E_P\left[ \left(\sum_{i=1}^n\prod_{j=0}^{i-1} R_{-j}\cdot \one(N_1=-n) \right)^s \right]
\\
&\leq&
\sum_{n=1}^\infty n^s \sum_{i=1}^n E_P\left[\prod_{j=0}^{i-1} R_{-j}^s\cdot \one(N_1= -n) \right]
\\
&=&
\sum_{n=1}^\infty n^s \sum_{i=1}^n \frac{1}{Q(R_0=1/a)}\cdot E_Q\left[\prod_{j=0}^{i-1} R_{-j}^s\cdot \one(N_1= -n,R_0=1/a) \right]
\\
&=&
\sum_{n=1}^\infty n^s \sum_{i=1}^n E_P\left[\prod_{j=n-(i-1)}^n R_j^s\cdot \one(T=n) \right]=\veps a^{-s}\sum_{n=1}^\infty n^s \sum_{i=1}^n \Theta^{n-i}\Theta_s^{i-1}h(1)<\infty,
\feq
where the last inequality is an implication of the fact that the spectral radius of the infinite positive matrix $\Theta_s$ is strictly less than one.
The proof of the lemma is complete.
\end{proof}
\begin{remark}
\label{essay}
The initial conditions $X_0=1$ and $X_1=a$ guarantee that the measure $P$ is $Q$ conditioned on the event $R_0=1/a.$ If $R_0$ has a different value, then $A_0$ and $B_0$ defined above
are still independent of the i.\,i.\,d. sequence of pairs $(A_n,B_n)_{n\in\nn},$ but the distributions of the pairs $(A_0,B_0)$ and $(A_1,B_1)$ differ in general. Using a slightly more elaborated version of the arguments used in the proof of Lemma~\ref{stablef} (cf. proof of Proposition~2.38 under assumption (1.6) in \cite{stable}) it can be shown that all the conclusions of Theorem~\ref{main} remain valid for different strictly positive initial values $(X_0,X_1)$ and that the only effect of changing initial conditions is on the value of the constant $K_\veps.$
\end{remark}
To conclude the proof of Theorem~\ref{main} it remains to prove the claim in part {\em (b)} of the theorem. Using the representation of $\Lambda_\veps(t)$ given in \eqref{istable} and a
variation of the coupling argument which we employed in order to prove Proposition~\ref{renewal}, we first derive the following auxiliary result:
\begin{lemma}
\label{lambda-s}
For any fixed $t>0,$ $\Lambda_\veps(t)$ is a strictly increasing function of $\veps$ on $[0,1].$
\end{lemma}
\begin{proof}
Recall the notation introduced in the course of the proof of Proposition~\ref{renewal}. It follows from the inequality in \eqref{couplee} that
\beq
\frac{1}{X^{(1)}_n}&\geq& \frac{1}{X_n}\cdot\prod_{k=0}^{n-2} \Bigl(1+\frac{b}{a^2+b}\Bigr)^{\zeta_k}
=\frac{1}{X_n}\cdot\exp\left\{\sum_{k=0}^{n-2} \zeta_k \cdot \log \Bigl(1+\frac{b}{a^2+b}\Bigr)\right\}.
\feq
It follows then from H\"{o}lder's inequality that for any constants $t>0,$ $p>1$ and $q>0$ such that $\frac{1}{p}+\frac{1}{q}=1,$ we have
\beq
E_P\left(\frac{1}{X_n^t}\right)&\leq& \left[E_P\left(\frac{1}{\bigl(X_n^{(1)}\bigr)^{pt}}\right)\right]^{1/p}\cdot \left[E_P\Bigl(\exp\Bigl\{-qt\sum_{k=0}^{n-2} \zeta_k \cdot \log \Bigl(1+\frac{b}{a^2+b}\Bigr)\Bigr\}\Bigr)\right]^{1/q}.
\feq
Therefore, since $\zeta_k$ are i.\,i.\,d. Bernoulli random variables,
\beq
\Lambda_{\veps}(t)&\leq& \frac{1}{p}\cdot \Lambda_{\veps_1}(pt)+\limsup_{n\to\infty}\frac{1}{nq}\log E_P\Bigl(\exp\Bigl\{-qt\sum_{k=0}^{n-2} \zeta_k \cdot \log \Bigl(1+\frac{b}{a^2+b}\Bigr)\Bigr\}\Bigr)
\\&=&
\frac{1}{p}\cdot \Lambda_{\veps_1}(pt)+\frac{1}{q}\log E_P\Bigl(\exp\Bigl\{-qt\zeta_1 \cdot \log \Bigl(1+\frac{b}{a^2+b}\Bigr)\Bigr\}\Bigr)
\\
&=&
\frac{1}{p}\cdot \Lambda_{\veps_1}(pt)+\frac{1}{q}\log\Bigl[(\veps_1-\veps)\cdot \exp\Bigl\{-qt\log \Bigl(1+\frac{b}{a^2+b}\Bigr)\Bigr\}+(1-\veps_1+\veps)\cdot 1\Bigr]
\\
&=&
\frac{1}{p}\cdot \Lambda_{\veps_1}(pt)+\frac{1}{q}\log\Bigl[(\veps_1-\veps)\cdot \Bigl(1+\frac{b}{a^2+b}\Bigr)^{-qt}+(1-\veps_1+\veps)\Bigr]
\\
&\leq&
\frac{1}{p}\cdot \Lambda_{\veps_1}(pt)-\frac{1}{q}(\veps_1-\veps)\Bigl[1- \Bigl(1+\frac{b}{a^2+b}\Bigr)^{-qt}\Bigr],
\feq
where in the last step we used the inequality $\log(1-x)<x.$
Since $\Lambda_\veps(t)$ is a continuous function of $t,$ by letting $p$ to approach one and thus $q$ to approach infinity, we obtain that
\beq
\Lambda_{\veps_1}(t)-\Lambda_\veps(t)\geq t(\veps_1-\veps)\log \Bigl(1+\frac{b}{a^2+b}\Bigr)>0.
\feq
The proof of the lemma is complete.
\end{proof}
We now turn to the proof of part {\em (b)} of Theorem~\ref{main}.
\begin{proposition}
\label{prop3}
The critical exponent $s_\veps$ is a strictly decreasing continuous function of $\veps$ on $[0,\veps^*).$
Furthermore, \eqref{limits} holds true.
\end{proposition}
\begin{proof}
The desired monotonicity of $s_\veps$ follows directly from Lemma~\ref{lambda-s}, see Fig.~\ref{figure} above.
We will next show that $s_\veps$ is a continuous function of $\veps$ on $(0,1).$ Due to the monotonicity of
$s_\veps,$ the following one-sided limits exist for any $\veps\in (0,\veps^*):$
\beq
s_\veps^+=\lim_{\delta\downarrow\veps} s_\delta \qquad \mbox{\rm and} \qquad s_\veps^-=\lim_{\delta\uparrow\veps} s_\delta.
\feq
The second limit, namely $s_{\veps^*}^-,$ exists also for $\veps=\veps^*.$ Set $s_{\veps^*}:=0.$ If either $s_\veps^+>s_\veps$ or $s_\veps^- <s_\veps$ for some $\veps\in[0,\veps^*],$ then (see Fig.~\ref{figure} above) $\Lambda_\delta(t^*)$ is not a continuous function of $\delta$ at any point $t^*$ within the open
interval $(s_\veps,s_\veps^+)$ or, respectively, $(s_\veps^-,s_\veps).$ To verify the continuity of $s_\veps$ on $(0,\veps^*]$ it therefore suffices
to show that $\Lambda_\veps(t)$ is a continuous function of $\veps$ for any fixed $t>0.$
\par
We will use again the notation and the coupling construction introduced in the course of the proof of Proposition~\ref{renewal}.  Recall $T_n$ from \eqref{tin}, and let $T_n^{(1)},$ $n\in\nn,$ be the corresponding stopping times associated with the sequence $X_n^{(1)}.$ Let $\chi_n=\one(T_n\neq T_n^{(1)}).$ The random variables $\chi_n$ form a two-state Markov chain with transition kernel determined by
\beq
P(\chi_{n+1}=1|\chi_n=0)=P\bigl(\eta_n=\eta_n^{(1)}=0\bigr)=P(\zeta_{n+1}=0)=\veps_1-\veps
\feq
and
\beq
P(\chi_{n+1}=1|\chi_n=1)=1-P\bigl(\eta_n=\eta_n^{(1)}=0\bigr)=1-\veps.
\feq
The stationary distribution $\pi=\bigl(\pi(0),\pi(1)\bigl)$ of this Markov chain is given by
\beq
\pi(0)=\frac{\veps}{\veps_1}\qquad \mbox{\rm and} \qquad \pi(1)=\frac{\veps_1-\veps}{\veps_1}
\feq
Similarly to \eqref{couplee}, in virtue of Lemma~\ref{bounds} we have:
\beq
\frac{1}{X^{(1)}_n}&\leq& \frac{1}{X_n}\cdot\prod_{k=0}^{n-2} \Bigl(\frac{a^2+b}{a}\cdot \frac{1}{a}\Bigr)^{\xi_k}
=\frac{1}{X_n}\cdot\exp\left\{\sum_{k=0}^{n-2} \chi_k \cdot \log \Bigl(1+\frac{b}{a^2}\Bigr)\right\}.
\feq
It follows then from H\"{o}lder's and Jensen's inequalities that for any constants $t>0,$ $p>1$ and $q>0$ such that $\frac{1}{p}+\frac{1}{q}=1,$ we have
\beq
E_P\left(\frac{1}{\bigl(X^{(1)}_n\bigr)^t}\right)&\leq& \left[E_P\left(\frac{1}{X_n^{pt}}\right)\right]^{1/p}\cdot \left[E_P\Bigl(\exp\Bigl\{qt\sum_{k=0}^{n-2} \chi_k \cdot \log \Bigl(1+\frac{b}{a^2}\Bigr)\Bigr\}\Bigr)\right]^{1/q}
\\
&\leq&
\left[E_P\left(\frac{1}{X_n^{pt}}\right)\right]^{1/p}\cdot \left[\exp\Bigl\{qt\sum_{k=0}^{n-2} E_P(\chi_k) \cdot \log \Bigl(1+\frac{b}{a^2}\Bigr)\Bigr\}\right]^{1/q}.
\\
&=&
\left[E_P\left(\frac{1}{X_n^{pt}}\right)\right]^{1/p}\cdot \Bigl(1+\frac{b}{a^2}\Bigr)^{t\sum_{k=0}^{n-2} E_P(\chi_k)}.
\feq
Since Markov chain $\chi_n$ is aperiodic, its stationary distribution $\pi$ is the limiting distribution. Thus,
\beq
\Lambda_{\veps_1}(t)&\leq& \frac{1}{p}\cdot \Lambda_\veps(pt)+t\log \Bigl(1+\frac{b}{a^2}\Bigr)\lim_{n\to\infty} P(\xi_k=1)
\\
&=&
\frac{1}{p}\cdot \Lambda_\veps(pt)+t\pi(1)\cdot \log \Bigl(1+\frac{b}{a^2}\Bigr)=\frac{1}{p}\cdot \Lambda_\veps(pt)+t\frac{\veps_1-\veps}{\veps_1}\log \Bigl(1+\frac{b}{a^2}\Bigr).
\feq
Since $\Lambda_\veps(t)$ is a continuous function of $t$ and $p>1$ is arbitrary, we conclude that
\beq
0<\Lambda_{\veps_1}(t)-\Lambda_\veps(t)\leq t\frac{\veps_1-\veps}{\veps_1}\log \Bigl(1+\frac{b}{a^2}\Bigr),
\feq
and thus, for a given $t>0,$ $\Lambda_\veps(t)$ is a Lipschitz function of the parameter $\veps$ on any interval bounded away from zero. This completes the proof of the continuity of
$s_\veps$ on $(0,\veps^*].$ In particular, the second limit in \eqref{limits} holds true.
\par
To complete the proof of the proposition it remains to prove that the first limit in \eqref{limits} holds true, namely
$\lim_{\veps\downarrow 0} s_\veps=\infty.$ To this end it suffices to show that $s_\veps>t$ for all $\veps>0$ small enough.
To this end, observe that since \eqref{ze} implies $\lim_{n\to\infty} S_n=\lambda_1^{-1}<1,$ there exists $k_0\in\nn$ such that
$S_k<\frac{1}{2}\bigl(1+\lambda_1^{-1}\bigr)<1$ for all $k>k_0.$ For $n\in\nn,$ let $\delta_n=\one\bigl(R_n=S_k~\mbox{\rm with}~k> k_0\bigr)$ and let
$\calg_n=\sigma(R_1,R_2,\ldots,R_n)$ be the $\sigma$-algebra generated by the random variables $R_i$ with $1\leq i\leq n.$
Then, with probability one, we have for $n\geq 2,$
\beqn
\nonumber
P\bigl(\delta_n=0\bigl|\calg_{n-1}\bigr)&\leq& P\Bigl(\bigcup_{0\leq k\leq k_0} \{\eta_{n-k-2}=0\}\Bigl|\calg_{n-1}\Bigr)
\\
\label{key}
&\leq&
\sum_{k=0}^{k_0} P\bigl(\eta_{n-k-2}=0\bigl|\calg_{n-1}\bigr)=(k_0+1)\veps.
\feqn
Denote $u=\frac{1}{2}\bigl(1+\lambda_1^{-1}\bigr)$ and $v=a^{-1}.$ It follows from \eqref{istable}, \eqref{key},
and \eqref{ordering} that for $\veps<(1+k_0)^{-1}$ we have:
\beq
\Lambda_\veps(t)&=&\lim_{n\to\infty} \frac{1}{n}\log E_Q\Bigl(\prod_{i=1}^n R_i^t\Bigr)\leq
\limsup_{n\to\infty} \frac{1}{n}\log E_Q\Bigl(\prod_{i=1}^n u^{t\sigma_i}v^{t(1-\sigma_i)}\Bigr)
\\
&\leq&
\log \bigl[u^t\bigl(1-\veps(k_0+1)\bigr) +v^t\veps(k_0+1)\bigr].
\feq
It thus holds that $\Lambda_\veps(t)<0,$ and hence $s_\veps>t,$ for all $\veps>0$ small enough. Since $t>0$ is arbitrary, it follows that $\lim_{\veps \downarrow 0} s_\veps=\infty.$ The proof of the proposition is complete.
\end{proof}
\section{Proof of Proposition~\ref{stat}}
\label{stats}
For $k\in\nn,$ et $\calf_{k-1}=\sigma(X_0,\eta_0,X_1,\eta_1,\ldots,X_{k-2},\eta_{k-2},X_{k-1},\eta_{k-1},X_k,X_{k+1})$ be the $\sigma$-algebra generated by the random variables $\eta_i$ with $i\leq k-1$ and $X_i$ with $i\leq k+1.$ It follows from \eqref{xdef} that $\eta_k$ is independent of $\calf_{k-1}.$ In the proof below, we will repeatedly use
without further notice the fact $E_P(X\eta_k)=E_P[X(1-\veps)]$ for a random variable $X\in \calf_{k-1}.$
\begin{proof}
\item [\em (a)] In order to verify the claim, take the expectation on both sides of \eqref{xdef} and recall \eqref{lambdas}, \eqref{ze}.
\item [\em (b)] Take the square and then take the expectation on the both sides of \eqref{xdef}, to obtain:
\beqn
\nonumber
E_P(X_{n+1}^2)&=&a^2E_P(X_n^2)+b^2(1-\veps)^2E_P(X_{n-1}^2)+2b(1-\veps) E_P(aX_nX_{n-1})
\\
&=&
\nonumber
a^2E_P(X_n^2)+b^2(1-\veps)^2E_P(X_{n-1}^2)
\\
\nonumber
&&
\qquad\qquad
+2b(1-\veps) E_P\bigl[(X_{n+1}-b\eta_{n-1}X_{n-1})X_{n-1}\bigr]
\\
\nonumber
&=&
a^2E_P(X_n^2)-b^2(1-\veps)^2E_P(X_{n-1}^2)+2b(1-\veps) E_P(X_{n+1}X_{n-1})
\\
\nonumber
&=&
\bigl(a^2+2b(1-\veps)\bigr)E_P(X_n^2)-b^2(1-\veps)^2E_P(X_{n-1}^2)
\\
\label{sq}
&&
\qquad\qquad
+2b(1-\veps) E_P(h_n),
\feqn
where $h_n:=X_{n-1}X_{n+1}-X_n^2.$
\par
We will next derive a Cassini-type formula for $E_P(h_n).$ We have:
\beq
aE_P(h_{n+1})&=&E_P\bigl[(aX_n)\cdot  X_{n+2}\bigr]-aE_P\bigl(X_{n+1}^2\bigr)
\\
&=&
E_P\bigl[(X_{n+1}-b\eta_{n-1}X_{n-1})\cdot (aX_{n+1}+b\eta_nX_n)-aX_{n+1}^2\bigr]
\\
&=&
E_P\bigl[b\eta_nX_nX_{n+1}-ab\eta_{n-1}X_{n-1}X_{n+1}-b^2\eta_{n-1}\eta_nX_{n-1}X_n\bigr]
\\
&=&
E_P\bigl[b(1-\veps) X_nX_{n+1}-ab\eta_{n-1}(X_n^2+h_n)-b^2(1-\veps) \eta_{n-1}X_{n-1}X_n\bigr]
\\
&=&
E_P\bigl[b(1-\veps) X_n(X_{n+1}-aX_n-b\eta_{n-1}X_{n-1})-ab\eta_{n-1}h_n\bigr]
\\
&=&
-abE_P(\eta_{n-1}h_n).
\feq
Hence,
\beqn
\nonumber
E_P(X_nX_{n+2}-X_{n+1}^2)&=&E_P(h_{n+1})=-bE_P(\eta_{n-1}h_n)=\ldots
\\
&=&
\label{cassini}
(-b)^n\veps^{n-1}E_p(\eta_0h_1)=(-1)^nb^{n+1}(1-\veps)^n.
\feqn
Using the notation $Y_n=E_P(X_n^2)$ and substituting \eqref{cassini} into \eqref{sq}, we obtain
\beq
Y_{n+1}=\bigl[a^2+2b(1-\veps)\bigr]Y_n-b^2(1-\veps)^2Y_{n-1}+2(-b)^{n+1}(1-\veps)^n,
\feq
from which the claim in {\em (b)} follows, taking in account that $Y_0=1$ and $Y_1=a^2.$
\item [\em (c)] For any $k\in\nn,$ we have:
\beq
U_{n,k+1}&=&E_P(X_nX_{n+k+1})=aE_P(X_nX_{n+k+1})+bE_P(\eta_{n+k-1}X_nX_{n+k-1})
\\
&=&
aU_{n,k}+b(1-\veps)U_{n,k-1}.
\feq
Furthermore, using notations introduced in the course of proving {\em (b)},
\beq
E_P(X_nX_{n+1})&=& \frac{1}{a}E_P(X_nX_{n+2})-\frac{b(1-\veps)}{a}E_P(X_n^2)
\\
&=&
\frac{1}{a}\bigl[E_P(h_{n+1})+Y_{n+1}^2-b(1-\veps)Y_n^2\bigr].
\feq
The proof of the proposition is complete.
\end{proof}
\section{Concluding remarks}
\label{remarks}
\begin{enumerate}
\item We believe that $K_\veps$ in the statement of Theorem~\ref{main} is decreasing as a function of the parameter $\veps,$ but were unable to prove it.
Some information about this constant can be derived from the formulas given in \cite{darek, goldie} (see also references in \cite{darek}) using the recursion representation \eqref{recur} of $W_n$ and the regeneration structure described in Section~\ref{proofm} (see the proof of Lemma~\ref{stablef} there) which reduces the Markov setup of this paper to an i.\,i.\,d. one considered in \cite{darek, goldie, kesten}.
\item We think that $s_\veps$ is a strictly convex function of $\veps$ on $[0,\veps^*),$ but were unable to prove it.
Since $\lim_{\veps\downarrow 0} s_\veps=+\infty,$ Fig.~\ref{figure} strongly suggests that the convexity holds for an interval of small enough values of $\veps$
within $(0,\veps^*).$ We believe that, with $s_{\veps^*}$ set to zero, $s_\veps$ is convex in fact on the whole interval $(0,\veps^*].$
\item The linear model \eqref{xdef} can serve as an ansatz in a general case. For instance, it seems plausible that a result similar to our Theorem~\ref{main} holds for generalized Fibonacci sequences considered in \cite{fibbo}. This is a work in progress by the authors.
    \item Using appropriate variations of the above Proposition~\ref{prop} and Lemma~\ref{stablef}, Theorem~\ref{main} can be extended to a class of recursions
    $\witi W_{n+1}=\theta \cdot \prod_{i=0}^{m-1}R_{nl+i}^{h_i}\witi W_n+Q_n$
with arbitrary $l,m\in\nn,$ positive reals $h_i,$ and suitable coefficients $\theta$ (large enough by absolute value constant) and $Q_n$ (in general random).
For instance, in the spirit of \cite{fbook}, one can consider sequences $\witi W_n=\frac{1}{X_{2n}^2}\sum_{k=0}^{n-1}X_{2k+1}X_{2k+8}$ or
$\witi W_n=\frac{1}{X_n^2}\sum_{k=0}^{n-1} (-1)^kX_k^2.$ The former case corresponds to
$\witi W_{n+1}=Q_n\witi W_n+Q_n$ with $Q_n=R_{2n+1}R_{2n+2}R_{2n+8}R_{2n+9},$ and the later to $\witi W_{n+1}=Q_n\witi W_n+1$ with $Q_n=(-1)^nR_n^2.$ We leave details to the reader.
\end{enumerate}
\bibliographystyle{amsplain}

\end{document}